\documentclass[12pt, reqno, a4paper]{amsart}

\usepackage{amssymb}
\usepackage[bookmarks=false]{hyperref}

\theoremstyle{plain}

\newtheorem*{corollary}{Corollary}
\newtheorem{lemma}{Lemma}
\newtheorem{theorem}{Theorem}

\theoremstyle{remark}
\newtheorem*{remark}{Remark}

\theoremstyle{definition}

\newtheorem{example}{Example}

\begin{document}

\title{On semi-homogeneous maps of degree $k$}

\author{Nina A. Erzakova}

\address{Moscow State Technical University of Civil Aviation}
\email{naerzakova@gmail.com}

\begin{abstract} We study
properties of  continuous semi-homogeneous operators of degree $k$ via various functions (e.g. measures of noncompactness) on all bounded subsets of a Banach space.
We prove necessary and sufficient conditions for
 these functions to vanish on the image
   of the unit ball under these operators.
   In particular, we give criteria
 for superposition operators to be improving
 and criteria for the complete continuity of the Fr\'echet derivatives. The results obtained can be applied in various areas of both pure and applied mathematics. 
\end {abstract}



\medskip

\maketitle

\section{Introduction}

 Let  $E$, $E_1$ be  Banach spaces. Denote by $\theta$ the zero element of a Banach space, by  $S_r=\{u\in E:\|u\|=r\}$, $B_\rho=\{u\in E:\|u\|\leqslant\rho\}$ the  sphere and the ball of radius $r$ with the center at $\theta$, respectively.
  Let $\mathfrak{B}$ be the family of all subsets of a Banach space $E$ and
  let $\psi\colon \mathfrak{B}\to [0,+\infty)$ be some function.
     
 For example, $\psi$ can be the diameter of a subset of $E$ or
any  measure of noncompactness (an MNC for brevity).
 
 Numerous examples  of numerical characteristics of subsets of a Banach space $E$
 can be found  in \cite {Banas}, \cite{Akhmer} and \cite{Ayerbe}.
 
 We consider the class of all continuous operators
 $T\colon  E\rightarrow E_1$  (not necessarily  linear) satisfying the property
\begin{equation}
 \exists\  k=k(T)>0 :\  \psi(T(\rho U))=\rho^k\psi(T(U)) \ \forall U\subset E,\forall \rho >0,\label{R1}
 \end{equation}
 i.e. all {\it semi-homogeneous maps of degree $k$}.
 
We are interested in obtaining necessary and sufficient conditions for $\psi(T(B_1))=0$.
This equality has different meaning for $T$ depending on $\psi$ employed.
  
First recall basic definitions and notation~(\cite [3.1.2]{Banas}, \cite[1.1.4]{Akhmer},\cite[2.3]{Ayerbe}) in a form convenient for us.

We recall that  numerical functions $\psi$ and $\psi_1$ are equivalent   in a space $E$ if there exist constants $c_1>0$ and $c_2>0$ such that 
$$c_1\psi_1(U)\leqslant\psi(U)\leqslant c_2\psi_1(U)$$
for all $U\subset E$.

Also we recall that a numerical function $\psi$ is regular if $\psi_E(U)=0$ if and only if the closure of $U$ is compact.

Sometimes we will assume that $\psi$  (which is not necessarily  regular)
satisfies some of the following properties:

\begin{enumerate}

\item\label{invariance under passage to the closure}
\textit{invariance under passage to the closure}:
$\psi(\overline U)=\psi(U)$;

\item\label{invariance under passage to the convex hull}
\textit{invariance under passage to the convex hull}:
$\psi(\mathop{\mathrm{co}}U)=\psi(U)$;

 \item\label{semi-additivity}\textit{semi-additivity}:
 $ \psi(U\cup V) = \max \{\psi(U),\psi(V)\}$;
 
 \item\label{semi-homogeneity}\textit{semi-homogeneity}:
 $ \psi(tU)=\vert t\vert\psi(U)$ ($t$ is a number); 
 
 \item\label{monotonicity}\textit{monotonicity}: $U_1\subseteq U$ implies $\psi(U_1)\leqslant \psi(U)$;
 
\item\label{algebraic semi-additivity}\textit{algebraic semi-additivity}:
 $\psi(U+V)\le\psi(U)+\psi(V)$,
  where $ U+V=\{u+v:u\in U, v\in V\}$; 
  
 \item\label{nonsingularity}\textit{nonsingularity}: $\psi$ is equal to zero
 on every one-element set;
  
  \item\label{invariance under translations}\textit{invariance under translations}:
 $\psi(U+u)=\psi(U)\ \ (u\in E)$.

 \end{enumerate}
 
 Let $U \subset E$. Then 
 \textit{the Hausdorff MNC} $\chi(U)$ of the set $U$ is the infimum of the
numbers $\varepsilon > 0$ such that $U$ has a finite $\varepsilon$-net in $E$; 
\textit{the MNC $\beta$}
is defined as the supremum of all numbers $r>0$ such that there exists an infinite 
sequence in $U$ with
$\|u_n-u_m\|\geqslant r$
 for every $n\ne m$.

The MNCs  $\chi$ and $\beta$ are regular, equivalent and
possess all the properties listed above
(see, for example, \cite[1.3]{Akhmer}).
 
 The  diameter of a subset of $E$ also
 possesses all the properties listed above
(see, for example, \cite[Remark 2.2]{Ayerbe}), but it is not regular.
 
  We will prove a criterion for $\psi(T(B_1))=0$ assuming only that $\psi$
 satisfies the properties (\ref{semi-homogeneity})--(\ref{invariance under translations})
 above.
 
 We apply this criterion to superposition operators
 and Fr\'echet derivatives.

\section{Main result}
 
 We say that a function $\psi$ (not necessarily  regular) satisfies {\itshape the spherical property} if for any continuous map $T\colon E\rightarrow E_1$ (not necessarily  linear)
  and any $\rho_1>0$
  \begin{equation}
  \psi( T(B_{\rho_1}))>0 \iff \exists\ 0<\rho_0\leqslant \rho_1:
  \psi(T(S_{\rho_0}))>0.\label{Sphere}
  \end{equation}
  
  \begin{remark}
  Note that if $\psi$ possesses  the  properties of invariance under passage to the convex hull, nonsingularity 
  and   semi-additivity  we have $\psi( B_\rho)=\psi( co(\{\theta\}\cup S_\rho)=\max \{\psi(\{\theta\}),\psi(S_\rho)\}=\psi( S_\rho)$ for every $\rho$.
 
Hence $\psi( TB_\rho)=\psi( TS_\rho)$ for every $\rho$ if  $T\colon E\rightarrow E_1$ is   linear operator.

Below we do not require that $\psi$ satisfies (\ref{invariance under passage to the closure})--(\ref{semi-additivity}).
In particular, $\psi$ is not necessarily an MNC in the sense of
(\cite [3.1]{Banas}, \cite[1.2]{Akhmer},\cite[2.1]{Ayerbe}).
  \end {remark}

 Consider the following example of a non-regular function $\psi$ possessing the spherical property:
 
 \begin{example}\label{Example8}
Following \cite{Krasnosel1} we consider the measure of nonequiabsolute continuity $\nu(U)$
of norms of elements of $U\subset L_p$ where $L_p=L_p(\Omega)$, $1\leqslant p<\infty$,
 $\Omega\subset \mathbb R^n$, $\mu(\Omega)<\infty$,
 $\mu$ is a \textit{continuous} measure  on~$\Omega$,
i.e. for every $D\subset\Omega$, $\mu(D) > 0$,
there exist $D_1, D_2 \subset D$, $D_1 \cup D_2 = D$,
$D_1 \cap D_2 = \varnothing$, $\mu(D_1)=\mu(D_2)=\frac{\mu(D)}{2}$. 
 
 Then the {\itshape  measure of nonequiabsolute continuity} $\nu(U)$ is defined by
$$
 \nu(U)= \ \ \mathop{\overline{\lim}}\limits_{\mu(D)\to 0}
 \ \ \sup\limits_{u\in U}\|P_Du\|_{L_p}, 
 $$
where $P_D \colon L_p \rightarrow L_p$ for any measurable $D \subseteq \Omega$ is  the operator $(P_D u)(s) = \left\lbrace\begin{array}{cl} u(s), & s\in D, \\
0, & x \notin D.  \end{array}\right.$ 
 
 We claim that $\nu$ possesses the spherical property.
 \end{example}
 \begin{proof}
  Let $T\colon L_p\rightarrow L_q$ ($1\leqslant p,q<\infty$)
 be a continuous operator.
Suppose $\nu(T(B_{\rho_1}))>0$ for some $\rho_1>0$.
By the definition of $\nu$, there exist a sequence of elements
$\{v_n\}$ belonging to $B_{\rho_1}$ and subsets $\{D_n\}$ such that 
$$\nu(T(B_{\rho_1}))=\lim\limits_{n\to\infty}\|P_{D_n}T(v_n)\|_{L_q}.$$
Since $[0,\rho_1]$ is a compact set in $\mathbb{R}^1$ there exists at 
least one limit point $\rho_0>0$ of the set $\{\|v_n\|_{L_p}\}\subset [0,\rho_1]$.
Thus we can extract a subsequence $\{\tilde v_n\}$
with $\lim\limits_{n\to\infty}\|\tilde v_n\|_{L_p}=\rho_0$.
Moreover, by the continuity of $T$ there exist  $\{u_n\}\subset S_{\rho_0}$ 
and subsets $\{\tilde D_n\}$ such that 
$$\nu(T(B_{\rho_1}))=\lim\limits_{n\to\infty}\|P_{\tilde D_n}T(u_n)\|_{L_q}.$$

Thus if $\nu(T(B_{\rho_1}))>0$ for some $\rho_1>0$ then there
exist $0<\rho_0\leqslant \rho_1$ and  a sequence $\{u_n\}$ belonging to $S_{\rho_0}$ 
such that $\nu(\{T(u_n)\})=\nu(T(B_{\rho_1}))$. 

By the monotonicity of $\nu$, the inequality $\nu(T(S_{\rho_0}))>0$
for some $\rho_0>0$
implies $\nu(T(B_{\rho_1}))>0$ for all $\rho_1\geqslant \rho_0$.
Thus  $\nu$ possesses the spherical property (\ref{Sphere}).
  
 We can  consider the  measure of nonequiabsolute  continuity   of norms of elements of $U$ in $E$ for any regular space E  of measurable functions not only $L_p$ (see, for example,   \cite{Erz14}).

The measure $\nu(U)$ has all properties of $\psi$ mentioned above, except the regularity, since the equality $\nu(U)=0$ is possible on noncompact sets too.
\end{proof}
  
 \begin{example}\label{Example9}
All MNCs $\psi$  equivalent to $\beta$  possess the spherical property (\ref{Sphere}).
\end{example}
\begin{proof} 
It is enough to prove  the   spherical property for $\beta$.  
Let $T$ be a continuous map $T\colon E\rightarrow E_1$ (not necessarily  linear).
  Suppose that for some $\rho_1>0$ 
 we have $\beta(T(B_{\rho_1}))>0$.
  
 By the definition of $\beta$, for every
$0<\varepsilon< \beta(T(B_{\rho_1}))/3$
  there exists an infinite sequence $\{u_n\}\subset U$ such that 
 $\|T(u_n)-T(u_m)\|>\beta(T(B_{\rho_1}))-\varepsilon$
for all $n\not=m$.

Since $[0,\rho_1]$ is a compact set in $\mathbb{R}^1$ there exists at 
least one limit point $\rho_0>0$ of the set $\{\|u_n\|\}\subset [0,\rho_1]$.
Thus 
we can extract a subsequence $\{v_n\}$
with $\lim\limits_{n\to\infty}\|v_n\|=\rho_0$.
Moreover, there exists  $\{\tilde v_n\}\subset S_{\rho_0}$ 
such that $\|T( v_n)-T(\tilde v_n)\|<\varepsilon$ by the continuity of $T$.

We have $$\beta(T(B_{\rho_1}))-\varepsilon<\|T(v_n)-T(v_m)\|=$$
$$= \|T(v_n)-T(v_m)-T(\tilde v_n)+T(\tilde v_n)-T(\tilde v_m)+T(\tilde v_m)\|\leqslant$$
$$\leqslant \|T(\tilde v_n)-T(\tilde v_m)\| + \|T( v_n)-T(\tilde v_n)\|+\|T( v_m)-T(\tilde v_m)\|\leqslant $$
$$\leqslant\|T(\tilde v_n)-T(\tilde v_m)\|+2\varepsilon$$ 
and $\beta(T(S_{\rho_0}))\geqslant\beta(\{T(\tilde v_n)\})>\beta(T(B_{\rho_1}))-3\varepsilon>0$.

By the monotonicity of $\beta$, the inequality $\beta(T(S_{\rho_0}))>0$
for some $\rho_0>0$
implies $\beta(T(B_{\rho_1}))>0$ for all $\rho_1\geqslant \rho_0$.

This completes the proof.

\end{proof}

\begin{lemma}\label{Lemma 1}
Let $\psi\colon \mathfrak{B}\rightarrow [0,+\infty)$ be a function possessing the spherical property.
Let $T\colon E\rightarrow E_1$ be a continuous operator
satisfying~(\ref{R1}) .
Then the equality
 $ \psi(T(S_{\rho_1}))=0 $ for some $\rho_1>0$
implies $ \psi(T(S_\rho))=\psi(T(B_\rho))=0 \ \forall \rho$.

\end{lemma}
\begin{proof} 

Let $T$ be be a continuous operator and
 $ \psi(T(S_{\rho_1}))=0 $ for $\rho_1>0$.

Suppose on the contrary that $ \psi(T(B_{\rho_1}))\not =0 $.
Then by the spherical property there exists a sphere $S_{\rho_0}$ with $0<\rho_0\leqslant \rho_1$ 
   such that $\psi(T(S_{\rho_0}))>0$.

This is contrary to the fact $0=\psi(T(S_{\rho_1}))=(\frac{\rho_1}{\rho_0})^k \psi(T(S_{\rho_0}))$
by (\ref{R1}).

Since we have reached contradiction, the equality $ \psi(T(B_{\rho_1}))=0 $ is proved.

By (\ref{R1})
the equalities
$ \psi(T(S_{\rho_1}))=\psi(T(B_{\rho_1}))=0$ for any $\rho_1>0$
imply $ \psi(T(S_\rho))=\psi(T(B_\rho))=0 \ \forall \rho$.

This completes the proof.

\end{proof}
 
 Fix a point $u_1\in M$.
  
 As in \cite{Erz16},  we
say that a continuous map  $f\colon M\subseteq E\rightarrow E_1$  is \textit{locally strongly  $\psi$-condensing operator at the point $u_1$},
if there exist a number $r_1>0$ and a  nonnegative function $\lambda_{u_1,f}$ on $[0,r_1]$,
   $\lim\limits_{r\to 0}\lambda_{u_1,f}(r)=0$, 
    such that 
$\psi_{E_1} (f(U))\leqslant \lambda_{u_1,f}(r) \psi_E (U)$
  for any  $0<\rho< r<r_1$ 
and  $U=(u_1+B_\rho)\cap M$. 

We  denote this class of operators  by ($\lambda_2$).

Also as in \cite{Erz16}, we
say that a continuous map   $f \colon G\subseteq E\rightarrow E_1$  is   \textit{strongly  $\psi$-condensing operator at infinity (on spherical interlayers)},  
if there exist a number $R_f>0$ and a  
nonnegative function $\tilde{\lambda}_{f}$  on $[R_f,\infty]$, 
$ \lim\limits_{r\to \infty}\tilde{\lambda}_{f}(r)=0$,
    such that
  $\psi_{E_1} (f(U))\leqslant \tilde{\lambda}_{f}(R_1) \psi_E (U)$ 
for any $R_2 > R_1>R_f$ and $U=G\cap (B_{R_2}\backslash B_{R_1})$.

We  denote this class of operators  by ($\tilde{\lambda}_2$).

 As in \cite{Erz15}, we
  say that a continuous map  $f\colon M\subseteq E\rightarrow E_1$  is
  \textit{locally strongly  $\psi$-condensing operator at the point $u_1$ (on spheres)}, 
if there exist a number $r_1>0$ and a  nonnegative function $\lambda_{u_1,f}$ on $[0,r_1]$,
   $\lim\limits_{r\to 0}\lambda_{u_1,f}(r)=0$,  
 such that 
$\psi_{E_1} (f(U))\leqslant \lambda_{u_1,f}(r) \psi_E (U)$
  for any  $0<\rho< r<r_1$ 
and  $U=(u_1+S_\rho)\cap M$.

We  denote this class of operators  by ($\lambda_3$).

Also as in \cite{Erz15}, we
say that a continuous map   $f \colon G\subseteq E\rightarrow E_1$  is \textit{strongly  $\psi$-condensing operator at infinity on spheres},  
if there exist a number $R_f>0$ and  
nonnegative function $\tilde{\lambda}_{f}$  on $[R_f,\infty]$, 
$    \lim\limits_{r\to \infty}\tilde{\lambda}_{f}(r)=0$,
    such that
  $\psi_{E_1} (f(U))\leqslant \tilde{\lambda}_{f}(R_1) \psi_E (U)$ 
for any $R_2 > R_1>R_f$ and $U=G\cap S_{R_2}$.

We  denote this class of operators  by ($\tilde{\lambda}_3$). 
 
  Consider the other class of continuous operators $f\colon  E\rightarrow E_1$ (not necessarily linear) which can be represented  in a neighborhood of  the point $u_1$ as the sum
 \begin{equation}
f(u_1+u)=A_1 (u) + A_0(u)
\label{L15} 
\end{equation} 
of a locally strongly  $\psi$-condensing operator $A_0\colon  E\rightarrow E_1$ at the point $\theta$ and a continuous operator
 $A_1\colon  E\rightarrow E_1$ (generally speaking depending of $u_1$) satisfying
 (\ref{R1}) by $ 0<k\leqslant 1$, i.e.
\begin{equation}
 \exists\  0<k\leqslant 1:\  \psi(A_1(\rho U))=\rho^k\psi(A_1(U)) \ \forall U\subset E,\forall \rho.\label{R15}
\end{equation}

 We denote the class of  $f$ satisfying (\ref{L15})  with $\psi (A_1(S_{\rho}))=0
 \ \forall \rho $  by ($\lambda_0$).
 
Consider also the class of continuous operators $f \colon  E\rightarrow E_1$ (not necessarily linear) which can be represented in a neighborhood of  infinity as the sum 
\begin{equation}
f(u)=\tilde{A}_1 (u) + \tilde{A}_0(u) 
\label{L16} 
\end{equation}
of an operator  $\tilde{A}_0\colon  E\rightarrow E_1$  from ($\tilde{\lambda}_2$)
 and a continuous operator $\tilde{A}_1 \colon  E\rightarrow E_1$  satisfying
 (\ref{R1}) by $ k\geqslant 1$, i.e.
\begin{equation}
 \exists\  k\geqslant 1:\  \psi(\tilde{A}_1(\rho U))=\rho^k\psi(\tilde{A}_1(U)) \ \forall U\subset E,\forall \rho.\label{R16}
\end{equation} 

 We denote the class of $f$ satisfying (\ref{L16})  
 with $ \psi(\tilde{A}_1(S_\rho))=0 \ \forall \rho $
  by ($\tilde{\lambda}_0$).

\begin{theorem}\label{Theorem 1}
Let $\psi\colon \mathfrak{B}\rightarrow [0,+\infty)$ be a function (not necessarily regular) possessing the spherical property
and the properties (\ref{semi-homogeneity})--(\ref{invariance under translations}).

(i) Let $u_1\in E$ be any point in the Banach space $E$ and $f\colon E\rightarrow E_1$ be  continuous operators (not necessarily linear).
Then the  classes 
$ (\lambda_0)$, $ (\lambda_2)$ and $ (\lambda_3)$
coincide.

(ii) Similarly,   the  classes 
$ (\tilde{\lambda}_0)$, $ (\tilde{\lambda}_2)$ and $ (\tilde{\lambda}_3)$
coincide.

\end{theorem}
 
 \begin{proof}  
 Note that by the monotonicity and the invariance of $\psi$ under translations we have 
 $\psi(f(u_1+S_\rho))\leqslant \psi(f(u_1+B_\rho))$ 
 and $\psi(u_1+S_\rho)=\psi(S_\rho)= \psi(B_\rho)=\psi(u_1+B_\rho)$ for all  $\rho>0$. Thus the class  $ (\lambda_3)$ includes
 the  class  $ (\lambda_2)$. Likewise the class  $ (\tilde{\lambda}_3)$ includes
 the  class  $ (\tilde{\lambda}_2)$.

 It remains to show that the class $(\lambda_2)$ includes the class $ (\lambda_0)$ and 
 the class $ (\lambda_0)$ includes the class $ (\lambda_3)$;
  the class $ (\tilde{\lambda}_2)$ includes the class $ (\tilde{\lambda}_0)$ and 
the class  $ (\tilde{\lambda}_0)$ includes the class $ (\tilde{\lambda}_3)$.
  
 (i) Let $f\in (\lambda_0)$. Then (\ref{L15}) implies
$f(u_1+B_\rho)\subseteq A_1(B_\rho)+A_0(B_\rho)$
for $\rho<r<r_2$ where $r_2$ is a positive number small enough.
From this, the monotonicity, and the algebraic semi-additivity of $\psi$ it follows 
 that 
$$\psi(f(u_1+B_\rho))\leqslant \psi(A_1(B_\rho))+\psi(A_0(B_\rho))\leqslant \psi(A_0(B_\rho)),$$
since $\psi(A_1(B_\rho))=0$ by
  the assumption $f\in (\lambda_0)$ and the spherical property of $\psi$
  and Lemma \ref{Lemma 1}.

Define $\lambda_{u_1,f}(r)=\lambda_{\theta,A_0}(r)$.

By the assumption, $A_0\colon  E\rightarrow E_1$ is a locally strongly  $\psi$-condensing operator at the point $\theta$. Then  
  $\lim\limits_{r\to 0}\lambda_{u_1,f}(r)=0$. Furthermore
 $$\psi(f(u_1+B_\rho))\leqslant \psi(A_0(B_\rho) )\leqslant \lambda_{\theta,A_0}(r)\psi(B_\rho)=\lambda_{u_1,f}(r)\psi(B_\rho),$$
 i.e., $f\in(\lambda_2)$ and $(\lambda_0)\subseteq (\lambda_2)$.
 
Let $f\in (\lambda_3)$. We shall prove now that $f\in (\lambda_0)$.

By (\ref{L15}) the following inclusion holds
$$
 A_1 (S_{\rho})\subseteq f(u_1+S_{\rho})-A_0(S_{\rho})$$
 for any $r_2<r_1$ and $0<\rho<r<r_2$.
 By the monotonicity and the algebraic semi-additivity of $\psi$ we have
 $$
 \psi(A_1 (S_{\rho}))\leqslant \psi(f(u_1+S_{\rho}))+\psi(A_0(S_{\rho})).$$
 The assumption (\ref{R15}) implies
$$\psi(A_1(S_1))\leqslant \psi(f(u_1+S_{\rho}))/{\rho^k}+\psi(A_0(S_{\rho}))/{\rho^k}.$$ 
 By  semi-homogeneity of $\psi$, we have $\psi(B_{\rho})=\psi(S_{\rho})=\rho\psi(S_1)$.
 
 By the definition of $(\lambda_3)$, the assumption of $A_0$ and
 $\psi(B_{\rho})=\psi(S_{\rho})=\rho\psi(S_1)$ we obtain
 $$
 \psi(A_1(S_1))\leqslant \lambda_{u_1,f}(r) \psi(S_{\rho})/\rho^k+\lambda_{\theta,A_0}(r)\psi(B_\rho)/\rho^k=$$
 $$=\psi(S_1)\lambda_{u_1,f}(r) \rho^{1-k}+\psi(S_1)\lambda_{\theta,A_0}(r)\rho^{1-k}$$
 for all $r$ small enough.

 Letting $r\to 0$ in the last inequality (we have consequently $\rho \to 0$)
 and taking into account that $ 0<k\leqslant 1$, we conclude
  $\psi(A_1 (S_1))=0$, since
 $\lambda_{u_1,f}(r) \rho^{1-k}\to 0$  and $\lambda_{\theta,A_0}(r)\rho^{1-k}\to 0 $ as $r \to 0$.
 Thus  $f$ belongs to $(\lambda_0)$ 
  and this completes the proof of the first part of the assertion.
  
 (ii)  Let $f\in (\tilde{\lambda}_0)$.
By the definition of (\ref{L16}) for any $R_2 > R_1>R_f$ 
 
  $f(B_{R_2}\backslash B_{R_1})\subseteq \tilde{A}_1(B_{R_2}\backslash B_{R_1}) + \tilde{A}_0(B_{R_2}\backslash B_{R_1})$.

The monotonicity and the algebraic semi-additivity 
 of $\psi$ imply 
$$\psi(f(B_{R_2}\backslash B_{R_1}))\leqslant \psi(\tilde{A}_1(B_{R_2}\backslash B_{R_1}))+\psi(\tilde{A}_0(B_{R_2}\backslash B_{R_1})).$$
Since $f\in (\tilde{\lambda}_0)$, then  $\psi(\tilde{A}_1(S_{R_2}))=0$
and  $\psi(\tilde{A}_1(B_{R_2}\backslash B_{R_1}))=0$ 
   by the spherical property of $\psi$
  and Lemma \ref{Lemma 1}.

Thus
$$\psi(f(B_{R_2}\backslash B_{R_1}))\leqslant  \tilde{\lambda}_{A_0}(R_1)\psi(B_{R_2}\backslash B_{R_1})$$
since $\tilde{A}_0\in (\tilde{\lambda}_2)$.
Putting, $\tilde{\lambda}_{f}(R_1)=\tilde{\lambda}_{\tilde{A}_0}(R_1)$ we get $f\in (\tilde{\lambda}_2)$  by the definition of $(\tilde{\lambda}_2)$.
 
Let $f\in (\tilde{\lambda}_3)$. 
 It follows from  (\ref{L16}) that
$$
 \tilde{A}_1(S_{R_2})\subseteq f(S_{R_2})-\tilde{A}_0(S_{R_2})$$
 for $R_2 > R_1>R_f$. 
 By the monotonicity and the algebraic semi-additivity of $\psi$,
 $$
\psi(\tilde{A}_1(S_{R_2}))\leqslant\psi (f(S_{R_2}))+\psi( \tilde{A}_0(S_{R_2})).
$$
 By the assumption (\ref{R16})
\begin{equation}
\psi(\tilde{A}_1(S_1))\leqslant\psi (f(S_{R_2}))/R_2^k+\psi( \tilde{A}_0(S_{R_2}))/R_2^k.
\label{Lambda13}
\end{equation}
 By the semi-homogeneity of $\psi$ we have
 $\psi(S_{R_2})=\psi(S_1)R_2$.
Since $f\in (\tilde{\lambda}_3)$, $\psi(S_{R_2})=\psi(S_1)R_2$,
 we obtain from (\ref{Lambda13}) 
$$\psi(\tilde{A}_1(S_1))\leqslant  \tilde{\lambda}_{f}(R_1)\psi (S_{R_2})/R_2^k +\tilde{\lambda}_{\tilde{A}_0}(R_1)\psi (S_{R_2})/R_2^k\leqslant  \psi(S_1)\tilde{\lambda}_{f}(R_1)R_2^{1-k} +$$
$$+\psi(S_1)\tilde{\lambda}_{\tilde{A}_0}(R_1)R_2^{1-k}.$$
 
 Taking into account that $ k\geqslant 1$ and letting $R_1\to \infty$ and consequently  $R_2 \to \infty$ in the last inequality, 
we get
 $\psi(\tilde{A}_1(S_1))=0$, and by (\ref{R16}), $\psi(\tilde{A}_1(S_\rho))=
 \rho^k\psi(\tilde{A}_1(S_1))=0$. Hence   $f\in (\tilde{\lambda}_0)$.

The theorem is proved.
 \end{proof}
\begin{corollary} 
Suppose in (\ref{L15}) we have $A_0(u) = 0$ for all $u
\in E$ (respectively, in (\ref{L16}) we have $\tilde A_0(u) = 0$ for all $u\in E$).
 Then we get the following necessary and sufficient conditions for $\psi (T(B_1))=0$ where $T$ is $A_1$ (respectively, $\tilde{A}_1$): 
 $\psi (A_1(B_1)))=0 \iff A_1 \in (\lambda_2)$,  
 $\psi (A_1(B_1)))=0 \iff A_1 \in (\lambda_3)$,
 $\psi (\tilde{A}_1(B_1)))=0 \iff \tilde{A}_1 \in (\tilde{\lambda}_2)$
 and
 $\psi (\tilde{A}_1(B_1)))=0 \iff \tilde{A}_1 \in (\tilde{\lambda}_3)$.
\end {corollary}
 
 Recall that an operator is \textit{completely continuous} if it is continuous and compact.
\begin{remark}
As a consequence of Theorem 1, in the case of a regular function $\psi$, we obtain two conditions equivalent to the complete continuity of operators $A_1$ and $\tilde{A}_1$ (not necessarily linear) having the representations (\ref{R15}), (\ref{R16}) respectively. 

If the spherical property of $\psi$ was not assumed in Theorem \ref{Theorem 1},
then the theorem would assert the following:

for $f$ satisfying (\ref{L15}), (\ref{R15})

$\psi(A_1(B_1))=0\iff  f\in (\lambda_2)$;\ \ \
$\psi(A_1(S_1))=0\iff  f\in (\lambda_3)$;

for $f$ satisfying (\ref{L16}), (\ref{R16})

$\psi(\tilde{A}_1(B_1))=0\iff  f\in (\tilde{\lambda}_2)$;\ \ \
$\psi(\tilde{A}_1(S_1))=0\iff  f\in (\tilde{\lambda}_3)$.

\end{remark} 

\section{Applications, comment}

We recall  (see, for example, \cite[Chapter II, 4.8]{Krasnosel2}) that an operator $f\colon E\to E$ is {\it  differentiable} at $u_1\in E$
if there exists a bounded linear operator $f^{\prime}(u_1)$,
called the {\it Fr\'echet derivative of $f$ at the point $u_1$},
such that
\begin{equation}
f(u_1+u)- f(u_1) = f^{\prime}(u_1)u+\omega(u)\label{L17}
\end{equation}
  where
\begin{equation}\lim\limits_{\|u\|\to 0}\frac{\|\omega(u)\|}{\|u\|}=0.
\label{R17}
\end{equation}

We denote the class of operators with the completely continuous Fr\'echet derivative at a point by ($\lambda_1$).

We recall  (see, for example, \cite[3.3.3]{Akhmer}) that an operator $f:E\to E$ is {\it asymptotically linear}
if there exists a bounded linear operator $f^{\prime}(\infty)$, 
called the {\it derivative of $f$ at the point $\infty$}, or the {\it asymptotic derivative of $f$} such that
\begin{equation}
f(u) = f^{\prime}(\infty)u+\tilde{\omega}(u),\label{L18}
\end{equation}
where
\begin{equation}\lim\limits_{\|u\|\to\infty}\frac{\|\tilde{\omega}(u)\|}{\|u\|}=0.
\label{R18}
\end{equation}

We denote the class of operators with the completely continuous Fr\'echet derivative at infinity (asymptotic derivative) by ($\tilde{\lambda}_1$).

 As an application of Theorem 1, we obtain criteria of complete continuity
 of the Fr\'echet derivative at a point and the asymptotic derivative.
 
 \begin{theorem}\label{Theorem 2}
 Let $\psi$ be an MNC equivalent to $\chi$.

(i) Let $f\colon E\rightarrow E$ be  an operator such that there 
exist  the continuous  Fr\'echet derivative $f^\prime(u_1)$ at $u_1\in E$.
Then for the given $u_1$
 the  classes 
$ (\lambda_1)$, $ (\lambda_2)$ and $ (\lambda_3)$
coincide.

(ii) Similarly,  if $f \colon  E\rightarrow E$  is an asymptotically linear operator,
then 
$ (\tilde{\lambda}_1)$, $ (\tilde{\lambda}_2)$ and $ (\tilde{\lambda}_3)$
coincide.
 
\end{theorem}
 
 \begin{proof}
   We reduce  the  assertion of Theorem 2 to  Theorem 1 supposing  $\psi=\chi$ without loss of generality.

 We may consider   bounded linear operators $f^{\prime}(u_1)$ and 
 $f^{\prime}(\infty)$ as particular instances of $A_1$ and 
 $\tilde{A}_1$ in (\ref{L15}) and (\ref{L16}) respectively.
 
Note that $\chi(f(u_1))=0$
since the set $\lbrace f(u_1) \rbrace$ consists of one point and is compact.

Hence for the given  $u_1$

$f \in (\lambda_2) \iff  f -f(u_1)\in (\lambda_2)$ and
$f \in (\lambda_3) \iff  f-f(u_1)\in (\lambda_3)$.

 Define $A_0$ as $\omega(u)$ from (\ref{L17}) and put $\lambda_{\theta,A_0}(r)=\sup\limits_{\|u\|\leqslant r}(\|\omega(u)\|/\|u\|)$.

It follows from (\ref{R17}) that
  $\lim\limits_{r\to 0}\lambda_{\theta,A_0}(r)=0$.
 Moreover, since $\chi(B_\rho)=\rho$ we have for any  sufficiently small $0<\rho<r$ 
 by the definition of $\chi$
 $$ \chi(A_0(B_\rho))=\chi(\omega(B_\rho))\leqslant\sup\limits_{\|u\|\leqslant \rho}\|\omega(u)\|\leqslant \sup\limits_{\|u\|\leqslant \rho}(\|\omega(u)\|/\|u\|)\rho$$
 $$\leqslant\sup\limits_{\|u\|\leqslant r}(\|\omega(u)\|/\|u\|)\rho= \lambda_{\theta,A_0}(r)\chi(B_\rho),$$
 i.e., $A_0$ is a locally strongly  $\chi$-condensing operator at the point $\theta$. 

Taking into account (\ref{L17}),
we summarize the above as $(\lambda_1)\subset  (\lambda_0)$.

Similarly, if we define $\tilde{A}_0$ as $\tilde{\omega}(u)$ from (\ref{L18})
and put $$\lambda_{\tilde{A}_0}(r)=\sup\limits_{\|u\|\geqslant r}(\|\tilde{\omega}(u)\|/\|u\|)$$
then by (\ref{R18})
$\lim\limits_{r\to \infty}\tilde{\lambda}_{\tilde{A}_0}(r)=0$ and 
for any sufficiently large $R_2 > R_1$ 
$$ \chi(\tilde{A}_0( B_{R_2}\backslash B_{R_1}))\leqslant\sup\limits_{\|u\|\geqslant R_1}\|\tilde{\omega}(u)\|\leqslant$$
$$\leqslant
 \sup\limits_{\|u\|\geqslant R_1}(\|\tilde{\omega}(u)\|/\|u\|)R_2= \tilde{\lambda}_{\tilde{A}_0}(R_1)\chi(B_{R_2}\backslash B_{R_1}),$$
i.e. $\tilde{A}_0\in (\tilde{\lambda}_2$).
 
The assertion of Theorem 2 follows from Theorem 1. 
 \end{proof}

Also as another application of Theorem 1 we obtain a criterion 
for an operator  acting in regular spaces to
be improving.

We recall (see \cite [17.7] {Krasnosel1}) that a superposition operator $F\colon
  L_q\rightarrow L_p$ acting in Lebesgue spaces 
 is called \textit{improving} if the measure of nonequiabsolute  continuity  $\nu(T(U))=0$ 
for every bounded $U\subset  L_q$. We extend this concept to all operators $f\colon E\to E_1$
acting in regular spaces $E$, $E_1$ of measurable functions.  
This notion arises whenever one investigate solvability of an equation in regular spaces (see for example \cite{Mediterr}).

We reformulate Theorem \ref{Theorem 1} for $\psi=\nu$ and regular spaces $E$, $E_1$ and obtain the following assertion:

\begin{theorem}\label{Theorem 3} Let $E$, $E_1$ be regular spaces.
 Let $\psi$ be $\nu$, i.e. the {\itshape measure of nonequiabsolute continuity}
of norms of subsets.

(i) 
Suppose that for  continuous operators $f$, $A_1$, $A_0$ acting from $E$ into $E_1$ the assumptions (\ref{L15}) and (\ref{R15})
are  satisfied for $\nu$ and some point $u_1 \in E$. 
Then for the given $u_1$ the operator
$A_1$ is improving if and only if $ f\in (\lambda_2)$ and
 if and only if $ f\in (\lambda_3)$.

(ii) Similarly,  suppose that for  continuous operators $f$, $\tilde{A}_1$, $\tilde{A}_0$
acting from $E$ into $E_1$ the assumptions (\ref{L16}) and (\ref{R16}) are satisfied for $\psi=\nu$. 
Then 
$\tilde{A}_1$ is improving if and only if $ f\in (\tilde{\lambda}_2)$ and
if and only if $ f\in (\tilde{\lambda}_3)$.
\end{theorem}

\begin{remark}
As it was observed in \cite{Erz11}, if two operators $F$ and $F_1$  acting from a set 
$G\subseteq E$ of a regular space $E$ into
a regular space $E_1$ are comparable on the set $G$, i.e.
 $\exists \ b_1 \in E_1$: $|(F(u))(s)|\leqslant
  |b_1(s)|+|(F_1(u))(s)|$
 $\forall u\in G$,
 then $\nu(F(U))\leqslant \nu(F_1(U))$.
 Obviously, if $F_1$ belongs to one of the four classes  $ (\lambda_2),(\lambda_3)$,  $(\tilde{\lambda}_2)$ and $(\tilde{\lambda}_3)$ with respect to $\nu$,
 then an operator $F$ comparable to $F_1$ belongs to the same class.
\end{remark}

We illustrate Theorem \ref{Theorem 3} with the example of the
superposition operator.

\begin{example}\label{Example10}
Let  $1\leqslant q,p<\infty$. 
  By \cite [Lemma 17.6]{Krasnosel1} for every superposition operator
 $F\colon
  L_q\rightarrow L_p$
there exists $a > 0$ such that $F$ is comparable with the operator $F_1$
  defined  by
 $F_1(u)(s)= a\cdot \mathop\mathrm{sgn}(u(s))|u(s)|^{q/p}$
 for $u\in L_q$.
 
Then $F_1$  is an operator satisfying (\ref{R15}) if $q\leqslant p$
and (\ref{R16}) if $p\leqslant q$. 
 \end{example}
\begin{proof}
Consider $U \subseteq u_1+B_r$. 
We have
$$\nu_{L_p}(F_1 (U))=\ \ \mathop{\overline{\lim}}\limits_{\mu(D)\to 0}
 \ \ \sup\limits_{u\in U}\|P_DF_1(u)\|_{L_p}
 =$$ $$\ \ \mathop{\overline{\lim}}\limits_{\mu(D)\to 0}
 \ \ \sup\limits_{u\in U}\|P_D(a\cdot \mathop\mathrm{sgn}(u(\cdot))| u(\cdot)|^{q/p})\|_{L_p}
 =$$ $$=\ \ \mathop{\overline{\lim}}\limits_{\mu(D)\to 0}
 \ \ \sup\limits_{u\in U}a\|P_D| u|^{q/p}\|_{L_p}$$
 and
  $$\nu_{L_p}(F_1(\rho U))=\ \ \mathop{\overline{\lim}}\limits_{\mu(D)\to 0}
 \ \ \sup\limits_{u\in U}a\|P_D|\rho u|^{q/p}\|_{L_p}=\ \ \rho^{q/p}\nu_{L_p}(F_1 (U)).$$
Thus $F_1$  is an operator satisfying (\ref{R15}) if $q\leqslant p$
and (\ref{R16}) if $p\leqslant q$. 
 \end{proof}

\begin{remark}
1) If we reformulate the corollary of Theorem\ref{Theorem 1} for the superposition operator $F_1$, then we obtain an analogue of \cite[Theorem 17.5]{Krasnosel1}:

 If $q\leqslant p$, then the superposition operator $F_1$ is improving if and only if  $ F_1\in (\lambda_3)$. If $p\leqslant q$, then the superposition operator $F_1$ is improving if and only if  $ F_1\in (\tilde{\lambda}_3)$. 

2) The following questions arise naturally:

Does there exist a function $\psi\colon \mathfrak{B}\to [0,+\infty)$ on all bounded subsets of
a Banach space that satisfies the properties (\ref{monotonicity})--(\ref{invariance under translations}) listed above but does not possess the spherical property?

If we do not require the existence of the representation 
 (\ref{L15}) at $u_1\in E$, does the class $ (\lambda_2)$ coincide with the class $(\lambda_3)$?

Similarly, if we do not require the existence of the representation 
  (\ref{L16}),  does the class $ (\tilde{\lambda}_2)$ coincide with the class $ (\tilde{\lambda}_3)$  ?

Answers to these questions are not known.

3) Theorem~\ref{Theorem 2}  combines two different results (\cite [Theorem 1]{Erz16} and \cite [Theorem 1]{Erz15}) containing necessary and sufficient conditions for complete continuity of the Fr\'echet derivative at a point and the asymptotic derivative, in the case of their existence.

Another remark is that a number of authors obtained only sufficient conditions for complete continuity of the Fr\'echet derivative~\cite{Krasnosel2, Melamed, Erz6}.
 
 Unlike \cite{Erz15} and \cite{Erz16}, we have obtained a criterion for
 $\psi(T(B_1))=0$ not only for $T$ being a Fr\'echet derivative, but also for $T$ belonging to a class of operators not necessarily linear.
 
The classes of operators $ (\lambda_2)$, $ (\lambda_3)$, $ (\tilde{\lambda}_2)$ and $ (\tilde{\lambda}_3)$ were introduced in \cite {Erz15} and \cite {Erz16} for an MNC $\psi$ as a modification of the definition given by R.D. Nussbaum \cite{Nussbaum} and of the definitions given by the author \cite{Erz6}.

Here we have extended these concepts to all functions $\psi\colon \mathfrak{B}\to [0,+\infty)$.

The classes  $ (\lambda_2), (\lambda_3)$,  $(\tilde{\lambda}_2)$ and $(\tilde{\lambda}_3)$ of not necessarily linear condensing $(k,\psi)$-bounded  operators include the so-called strongly $\psi$-condensing operators defined in \cite{Erz6} .

There are many examples of strongly $\psi$-condensing operators. Such operators
form a linear space. In particular, the Hammerstein operator
is a strongly  $\psi$-condensing operator~\cite{Erz6}.

Two different results
on bifurcation points from~\cite{Krasnosel2} were generalized in \cite{Erz15} and~\cite{Erz16}
to strongly $\chi$-condensing operators.

In~\cite{Erz6} the author also proved the existence theorem for a solution of an equation of the form  $u=f(u;\lambda)$, where $f$ is
 a locally strongly $\psi$-condensing operator. 

Maps close to operators $T$ satisfying (\ref {R1})  have been studied for a long time (see for example  \cite{Quintas}
and \cite{Ke}).
\end{remark}

\end{document}